\documentclass [11pt]{amsart}
\usepackage{graphicx,amssymb,amsfonts,latexsym,amsmath,amsthm}
\usepackage{epsfig}
\usepackage[english]{babel}
\usepackage[latin1]{inputenc}

\numberwithin{equation}{section}

\newtheorem{theorem}{Theorem}

\newtheorem{coro}[theorem]{Corollary}

\newtheorem{definition}[theorem]{Definition}

\newtheorem{prop}[theorem]{Proposition}

\theoremstyle{definition}

\newcommand{\RR}{{\mathbb R}}
\newcommand{\R}{{\mathbb R}}

\newcommand{\N}{{\mathbb N}}
\newcommand{\NN}{{\mathbb N}}
\newcommand{\A}{\mathcal{A}}

\newcommand{\al}{{\alpha}}
\newcommand{\gm}{{\gamma}}
\newcommand{\de}{{\delta}}
\newcommand{\ph}{{\varphi}}

\newcommand{\be}{{\beta}}

\newcommand{\lm}{{\lambda}}

\newcommand{\bs}[1]{\langle #1\rangle}

\newcommand{\ra}{{\rightarrow}}

\newcommand{\dd}{{\partial}}

\newcommand{\ab}[1]{\left( #1\right)}
\newcommand{\av}[1]{\left\vert #1\right\vert}
\newcommand{\aV}[1]{\left\Vert #1\right\Vert}

\newcommand{\qand}{{\quad\mathrm {and}\quad}}

\newcommand{\vol}{{\mathrm {vol}}}
\newcommand{\si}{{\sigma}}
\newcommand{\se}{\sigma_{ \mathrm{ess}}}

\newcommand{\SC}{(\mbox{SC}_\infty)}

\newcommand{\SI}{(\mbox{SI}_\infty)}

\DeclareMathOperator{\supp}{supp}

\title[]{Unbounded
Laplacians on Graphs:\\
Basic Spectral Properties and the Heat
Equation}
\author[]{Matthias Keller$^1$}
\author[]{Daniel Lenz$^2$}
\address{$^1$ Mathematisches Institut, Friedrich Schiller Universit\"at Jena, D-07743 Jena, Germany, m.keller@uni-jena.de}
\address{$^2$ Mathematisches Institut, Friedrich Schiller Universit\"at Jena,
D-07743 Jena, Germany, daniel.lenz@uni-jena.de, URL: http://www.analysis-lenz.uni-jena.de/
}

\begin{document}
\maketitle

\begin{abstract}We discuss Laplacians on graphs in a framework of regular Dirichlet forms. We focus on
  phenomena related to unboundedness of the Laplacians.  This includes (failure of) essential selfadjointness, absence of essential spectrum and stochastic incompleteness.
\end{abstract}





\section*{Introduction} The study of  Laplacians on graphs is  a well established topic of research (see e.g. the monographs \cite{Chu,Col} and references therein). Such operators can be seen as discrete analogues to Schr\"odinger operators. Accordingly their spectral theory has received quite some attention. Such operators also arise  as generators of symmetric Markov processes and  they appear in the study of heat equations on discrete structures. Recently, certain themes  related to unboundedness properties of such operators have become a focus of attention. These themes include
\begin{itemize}
\item definition of the operators and essential selfadjointness,
\item absence of essential spectrum,
\item stochastic incompleteness.
\end{itemize}
In this paper we want to
survey recent developments  and  provide some new results. Our principle goal is to make these topics accessible to non-specialists by providing a somewhat gentle and introductory discussion.
\medskip

Let us be more precise. We consider a graph with weights on edges and vertices. The weights can be seen to give a generalized vertex degree.
\medskip

There is an obvious  way to formally  associate a symmetric nonnegative operator to such a  graph.  If the generalized vertex degrees are uniformly bounded this operator is bounded and all formal expressions make sense. If the generalized vertex degrees are not uniformly bounded  already the definition of a self adjoint operator is an issue. This issue can be tackled by proving essential selfadjointness of the formal operator  on the set of functions with compact support. This was done for locally finite weighted graphs by Jorgensen in \cite{Jor} and  for locally finite graphs by Wojciechowski in \cite{Woj1} (see \cite{Woj2} as well)  and Weber in \cite{Web}. These results require local finiteness and do not allow for weights on the corresponding $\ell^2 $ space. As discussed by Keller/Lenz in \cite{KL1} it is possible to get  rid of the local finiteness requirement and to   allow for weighted spaces  by using   Dirichlet forms. The corresponding results give a nonnegative  selfadjoint (but not necessarily essentially selfadjoint)  operator in quite some generality and provide a criterion for essential selfadjointness covering the  earlier results of \cite{Jor,Web, Woj1}. These topics are discussed in Section~1.\medskip

Having an unbounded nonnegative operator at ones disposal one may then wonder about its basic spectral features. These basic features include the position of the infimum of the spectrum and the existence of essential spectrum. Both issues can be approached via isoperimetric inequalities. In fact, lower bounds for the spectrum have  been considered by  Dodziuk \cite{Dod0} and Dodziuk/Kendall \cite{DK}. For planar graphs explicit estimates for the isoperimetric constant and hence for the spectrum can be found for instance in \cite{HJL,HiShi,KP,Mo,Ura}. Triviality of the essential spectrum for general graphs has been considered by Fujiwara \cite{Fuj}. The corresponding results deal with bounded operators only. (They allow for unbounded vertex degree but then force boundedness of the operators by introducing weights on the corresponding $\ell^2$ space.)  Still, the methods can be used to provide lower bounds on the spectrum and prove emptiness of the essential spectrum for unbounded Laplacians as well. For locally finite graphs this has been done by Keller in \cite{Kel}. Here, we present a generalization of the results of \cite{Kel} to the general setting of regular Dirichlet forms. This generalization also extends  the results of \cite{Fuj, DK} to our setting. This is discussed in Section~5.
\medskip

Finally, we  turn to  a (possible) consequence  of  unboundedness in the study of the heat equation viz stochastic incompleteness.  Stochastic incompleteness describes the phenomenon that mass vanishes in a diffusion process.  While this may a priori not seem to be  connected to unboundedness, it turns out to be connected. This has already been observed by Dodziuk/Matthai  \cite{DM}  and Dodziuk \cite{Dod} in that they show stochastic completeness for certain bounded operators on graphs.
A somewhat more structural connection is provided  by our discussion  below.  For locally finite graphs stochastic completeness has recently been investigated by Weber in \cite{Web} and Wojciechowski \cite{Woj1}. In fact, Weber presents sufficient conditions and Wojciechowski gives a characterization of stochastic incompleteness. This characterization is inspired by corresponding work of Grigor'yan on manifolds \cite{Gri} (see work of Sturm \cite{Stu} for related results as well). As shown in \cite{KL1} this characterization can be extended to regular Dirichlet forms. Details are discussed in  Section~8. There, we also provide some further background extending \cite{KL1}. Let us mention that this circle of ideas is strongly connected to questions concerning uniqueness of Markov process with given generator as discussed by Feller in \cite{Fel} and Reuter in \cite{Reu}. We take this opportunity to  mention the  very recent survey \cite{Woj3}  of Wojciechowski giving a thorough discussion of stochastic incompleteness for manifolds and graphs (with edge weight constant to one).\medskip

While our basic aim is to study unbounded Laplacians we complement our results by characterizing boundedness of the Laplacians in question in Section~3.\medskip

For a related study of basic spectral properties in terms of generalized solutions we refer the reader to \cite{HK}. \medskip

The paper is organized as follows. In Section~1 we  introduce our operators and discuss basic properties. Section~2 contains a useful minimum principle and some of its consequences.  Boundedness of the Laplacians in question is characterized  in Section~3.
A useful tool, the so called co-area formulae are investigated in Section~4. They are used in Section~5 to provide an isoperimetric inequality which is then used to study bounds on the infimum of the (essential) spectrum. This allows us in particular to characterize emptiness of the essential spectrum. The connection to Markov processes is discussed in Section~7. A characterization of stochastic incompleteness is given in Section~8.

\section{Graph Laplacians and  Dirichlet forms} \label{Graph}

Throughout $V$ will be a countably infinite  set.

\subsection{Weighted graphs}
We will deal with weighted graphs with vertex set~$V$.  A symmetric weighted graph over $V$ is  a pair  $(b,c)$ consisting of  a map $  b : V\times V\longrightarrow [0,\infty)$ with $b(x,x) =0$ for all $x\in V$ and a map $c : V\longrightarrow [0,\infty)$
satisfying the following two properties:

\begin{itemize}
\item[(b1)] $b(x,y)= b(y,x)$ for all $x,y\in V$.

\item[(b2)] $\sum_{y\in V} b(x,y) <\infty$ for all $x\in V$.
\end{itemize}
Then $b$ is called the edge weight and $c$ is called  killing term.\medskip

We  consider  $(b,c)$ or rather  the triple  $(V,b,c)$ as a weighted  graph with vertex set $V$ in the following way: An $x\in V $ with
$c(x)\neq 0$ is  thought to be connected to the point $\infty$ by an edge
with weight $c(x)$. Moreover, $x,y\in V$ with $b(x,y)>0$ are thought to be
connected by an edge with weight $b(x,y)$.   Vertices $x,y\in V$ with $b(x,y)
>0$ are called neighbors.  More generally, $x,y\in V$ are called connected if
there exist $x_0,x_1,\ldots,x_n \in V$ with $b(x_i, x_{i+1}) >0$,
$i=0,\ldots, n$ and $x_0 = x$, $x_n = y$. This allows us to define connected
components of $V$ in the obvious way.\medskip

\noindent Two examples have attracted particular attention.\medskip

\noindent\textbf{Example (Locally finite graphs):} Let  $(V,b,c)$ be a weighted graph  with $c\equiv 0$ and  $b(x,y)\in\{0,1\}$ for all $x,y\in V$. We can then think of the $(x,y)\in V\times V$ with $b(x,y)=1$ as connected by an edge with weight $1$. The condition (b2) then implies that any $x\in V$ is connected to only finitely many $y\in V$. Such graphs are known as locally finite graphs. This is the class of examples studied in \cite{Kel,Fuj,Woj1,Web}.\medskip

\noindent\textbf{Example (Locally finite weighted graphs):} Let $(V,b,c)$ be a weighted graph with $c\equiv 0$ and $b$ satisfying
$$ \sharp \{ y : b(x,y)\neq 0 \}< \infty$$
for all $x\in V$. Then, $(V,b,c)$ is called a locally finite weighted graph. This is the class of examples studied in \cite{Dod,Jor}.

\subsection{Dirichlet forms on countable sets}
 Let $m$ be a measure on $V$ with full support
(i.e., $m$ is a map $m :V\longrightarrow (0,\infty)$). Then, $(V,m)$ is a measure space.   A particular example is given by $m\equiv 1$.  We will deal exclusively with real valued functions. Thus, $\ell^p (V,m)$, $0<p<\infty$, is defined by
$$ \{ u : V\longrightarrow \RR: \sum_{x\in V}  m(x)|u(x)|^p <\infty\}.$$
Obviously, $\ell^2 (V,m)$ is a Hilbert space with inner product $\langle \cdot,\cdot\rangle = \langle\cdot,\cdot\rangle_m$ given by
$$\langle u, v\rangle := \sum_{x\in V} m(x)u(x) v(x)\;\:\mbox{and norm}\;\:\aV{u}:=\langle u, u\rangle^\frac{1}{2}.$$
Moreover we denote by $\ell^\infty(V)$ the space of bounded functions on $V$. Note that this space does not depend on the choice of $m$. It is equipped with the supremum norm $\|\cdot\|_\infty$ defined by
$$\|u\|_\infty :=\sup_{x\in V} |u(x)|.$$
A symmetric nonnegative form on
$(V,m)$ is given by a dense  subspace $D$ of $\ell^2 (V,m)$ called the domain of the
form and a map
$$ Q : D\times D \longrightarrow \RR  $$
with $Q(u,v) = Q(v,u)$ and $Q(u,u)\geq 0$ for all $u,v\in D$.\medskip

Such a map is already determined by its values on the diagonal $\{(u,u) : u\in D\}\subseteq D\times D$. This motivates to consider the restriction of $Q$ to the diagonal as an object on its own right. Thus,  for $u\in
\ell^2 (V,m)$ we then define $Q(u)$ by
$$Q(u):= \left\{
           \begin{array}{ll}
             Q(u,u) &: u\in D, \\
             \infty &:  u\not\in D. \\
           \end{array}
         \right.
$$
If $\ell^2 (V,m)\longrightarrow [0,\infty]$,
$u\mapsto Q(u)$ is lower semicontinuous $Q$
is called closed.  If $Q$ has a closed extension it is called closable and the
smallest closed extension is called the closure of $Q$.\medskip

A map $C:   \RR\longrightarrow \RR $ with $C(0) =0$ and
$|C(x) - C(y)|\leq |x - y|$ is  called a normal contraction. If $Q$ is both
closed and satisfies $Q(Cu) \leq Q(u)$ for all normal contractions $C$ and all $u\in \ell^2 (V,m)$ it is
called a Dirichlet form on $(V,m)$ (see \cite{BH,Dav1,Fuk,MR} for background on   Dirichlet forms).\medskip

Let $C_c (V)$ be the space of finitely supported functions on $V$. A Dirichlet form on $(V,m)$ is called regular if its domain contains $C_c (V)$ and the form is the closure of its
restriction to the subspace $C_c (V)$. (The standard definition of regularity for Dirichlet forms  would require that $D(Q)\cap C_c (V)$ is dense in both $C_c (V)$ and $D(Q)$.  As discussed in \cite{KL1} this is equivalent to our definition.)

\subsection{From weighted graphs to Dirichlet forms}\label{FromWeightedGraphstoDF}
There is a one-to-one correspondence between weighed graphs and regular Dirichlet forms. This is discussed next.\medskip

To the weighted graph $(V,b,c)$ we can then  associate the form
$Q^{\rm{max}}= Q^{\rm{max}}_{b,c,m}:\ell^2 (V,m)\to [0,\infty]$
with diagonal given by
$$Q^{\rm{max}}(u) = \frac{1}{2}\sum_{x,y\in V} b(x,y) (u(x) - u(y))^2 + \sum_{x\in V} c(x) u(x)^2.$$
Here, the value $\infty$ is allowed. Let $Q^{\rm{comp}}=Q^{\rm{comp}}_{b,c}$ be the restriction of $Q^{\rm{max}}$ to $C_c (V)$. It is not hard to see that $Q^{\rm{max}}$ is closed. Hence $Q^{\rm{comp}}$ is closable on $\ell^2(V,m)$ and the closure will be denoted by $Q=Q_{b,c,m}$ and
its domain by $D(Q)$.\medskip

As discussed in \cite{KL1}  (see  \cite{Fuk} as well) the following holds.

\begin{theorem}\label{characterizationDF} The regular Dirichlet forms on $(V,m)$ are exactly given by
  the forms $Q_{b,c,m}$ with weighted graphs $(b,c)$ over $V$.
\end{theorem}

\noindent\textbf{Remark. } One may wonder whether the regularity assumption is necessary in the above theorem. It turns out that not every Dirichlet form $Q^{\rm{max}}_{b,c,m}$ is regular. A counterexample is provided in \cite{KL1}.\medskip

For a given a weighted graph $(V,b,c)$ the different  choices of measure $m$ will produce different Dirichlet forms. Two particular choices have attracted attention.  One is the choice of $m\equiv 1$. Obviously, this choice does not depend on $b$ and $c$. Another possibility is to use  $n=m= m_{b,c}$  given by
$$n(x) := \sum_{y\in V} b(x,y) + c(x).$$
The advantage of this measure is that it  produces a  bounded form (see below for details).

\subsection{Graph Laplacians}
Let $m$ be a measure on $V$ of full support,  $(b,c)$ a weighted graph over $V$ and $Q_{b,c,m}$ the  associated regular Dirichlet form. Then, there exists a unique selfadjoint operator $L = L_{b,c,m}$ on
$\ell^2 (V,m)$  such
that
$$ D(Q) :=\{ u\in \ell^2 (V,m) : Q(u)<\infty\} = \mbox{Domain of definition of $L^{1/2}$}$$
and
$$ Q(u) = \langle L^{1/2} u , L^{1/2} u\rangle$$
for $u\in D(Q)$ (see e.g.  Theorem 1.2.1 in \cite{Dav1}). As $Q$ is
nonnegative so is $L$.

\begin{definition} Let $V$ be a countable set and $m$ a measure on $V$ with full support.
A graph Laplacian on $V$ is an operator $L$ associated to a form $Q_{b,c,m}$.
\end{definition}

Our next aim is to describe the operator $L$ more explicitly: Define the formal Laplacian  $\widetilde{L} = \widetilde{L}_{b,c,m}$ on the vector space
\begin{equation}\label{ftilde}
\widetilde{F}:=\{ u : V\longrightarrow \RR : \sum_{y\in V} |b(x,y) u(y)|<\infty\;\mbox{for all
  $x\in V$ } \}
  \end{equation}
by
$$\widetilde{L} u (x) :=\frac{1}{m(x)} \sum_{y\in V} b(x,y) (u(x) - u(y)) +
\frac{c(x)}{m(x)}  u(x),$$
where, for each $x\in V$,  the sum exists  by assumption on $u$. The operator $\widetilde{L}$ describes the action of $L$ in the following sense.

\begin{prop}\label{generator} Let $(V,b,c)$ be a weighted graph and $m$ a  measure on $V$ of full support. Then, the operator  $L$ is a  restriction of $\widetilde{L}$ i.e.,   $$   D(L)\subseteq \{ u\in \ell^2 (V,m) : \widetilde{L} u \in \ell^2(V,m)\}\;\:\mbox{and}\;\: Lu = \widetilde{L} u $$
for all $u\in D(L)$.
\end{prop}

In order to obtain further information we need a stronger condition.  We define condition $(A)$ as follows:
\begin{itemize}
\item[$(A)$] For any sequence $(x_n)$ of vertices in  $V$ such that $b(x_n,x_{n+1})>0$ for all $n\in \NN$, the equality $ \sum_{n\in \NN} m (x_n) =\infty$ holds.
\end{itemize}

Let us emphasize that  in general  $(A)$ is a condition on $(V,m)$ and $b$
together. However, if
$$\inf_{x\in V}  m_x >0$$
holds, then obviously  $(A)$ is satisfied for all graphs $(b,c)$ over $V$.   This applies in particular to the case that $m\equiv 1$.\medskip

Given $(A)$ we can say more about the generators \cite{KL1}.

\begin{theorem}\label{generatorthm} Let $(V,b,c)$ be a weighted graph and $m$ a
 measure on $V$ of  full support such that $(A)$ holds. Then, the operator  $L$ is the
  restriction of $\widetilde{L}$ to
  $$   D(L)=\{ u\in \ell^2 (V,m) : \widetilde{L} u \in \ell^2
  (V,m)\}. $$
\end{theorem}

\noindent\textbf{Remark. } The theory of Jacobi matrices already provides examples
showing that without $(A)$ the statement becomes false \cite{KL1}.\medskip

The condition $(A)$ does not imply that $\widetilde{L} f$ belongs to
$\ell^2 (V,m)$ for all $f\in C_c (V)$.
However,  if this is the case, then $(A)$ does imply essential
selfadjointness. In this case, $Q$ is the ``maximal'' form associated to the
graph $(b,c)$. More precisely, the following holds \cite{KL1}.

\begin{theorem}\label{essential}
Let $V$ be a set,  $m$ a measure on $V$ with full support,  $(b,c)$ a graph
over $V$ and $Q$ the associated regular  Dirichlet form. Assume
$\widetilde{L} C_c (V) \subseteq  \ell^2 (V,m)$. Then,   $D(L)$ contains $C_c (V)$. If furthermore $(A)$ holds, then the restriction of $L$ to $C_c (V)$  is essentially selfadjoint and the domain of $L$  is given by
$$D (L) = \{u\in \ell^2 (V,m) : \widetilde{L} u\in \ell^2 (V,m)\}$$
and the associated form  $Q$ satisfies  $Q = Q^{\rm{max}}$ i.e., $$ Q(u) = \frac{1}{2} \sum_{x,y\in V} b(x,y) (u(x) - u(y))^2 + \sum_{x\in V} c(x)
u(x)^2$$
for all $u\in \ell^2 (V,m)$.
\end{theorem}

\noindent\textbf{Remark. }  Essential selfadjointness may fail if $(A)$ does not hold as can be seen by examples \cite{KL1}.\medskip

If $\inf_{x\in V} m_x >0$ then both  $(A)$ and $\widetilde{L} C_c (V)
\subseteq  \ell^2 (V,m)$ hold  for any graph $(b,c)$ over $V$.  We therefore obtain the following corollary.

\begin{coro} Let $V$ be a set and  $m$ a measure on $V$ with $\inf_{x\in V} m_x >0$.
Then,   $D(L)$ contains $C_c (V)$, the restriction of $L$ to $C_c (V)$  is essentially selfadjoint and the domain of $L$  is given by
$$D (L) = \{u\in \ell^2 (V,m) : \widetilde{L} u\in \ell^2 (V,m)\}$$
and the associated form  $Q$ satisfies  $Q = Q^{\rm{max}}$.
\end{coro}

\noindent\textbf{Remark. } The corollary includes the case that $m\equiv 1$ and  we recover the corresponding results of  \cite{DK,Woj1,Web} on  essential selfadjointness. (In fact, the  cited works also have additional restrictions on $b$ but this is not relevant here.)

\section{Minimum principle and consequences} \label{Minimum}
An important tool in the proofs of the results of the previous section is a minimum principle. This minimum principle shows in particular the relevance of $(A)$ in our considerations. This is discussed in this section.
\medskip

The following result is a variant and in fact a slight generalization  of the minimum principle from \cite{KL1}.

\begin{theorem} (Minimum principle)  Let  $(V,b,c)$ be a  weighted graph and $m$ a measure on $V$ of full support.  Let $U\subseteq V$
  be connected. Assume that the function  $u$ on  $V$ satisfies
\begin{itemize}
 \item $(\widetilde{L} + \alpha)  u \geq 0$ on  $U$ for some $\alpha>0$,
\item $u \geq 0$ on  $V\setminus U$.
\end{itemize}
Then,   the value of $u$ is nonnegative in any local minimum of $u$.
\end{theorem}
\begin{proof}  Let $u$ attain a local minimum on $U$ in $x_m$. Assume  $u(x_m)< 0$.
Then,  $u(x_m)\leq u(y)$ for all $y\in U$ with $b(x_m,y)>0$.  As $u(y)\geq 0$ for $y\in V\setminus  U$,
  we obtain   $u(x_m) - u(y) \leq
  0$ for all  $y\in V$ with $b(x_m,y)\geq 0$. By the super-solution assumption we find
$$ 0 \leq \sum b(x_m,y) (u(x_m) - u(y)) + c(x_m) u(x_m) + m(x_m)\alpha u(x_m)  \leq 0.$$
As $b$ and $c$ are nonnegative, $m$ is positive and $\alpha>0$, we  obtain the contradiction $0 = u(x_m)$.
\end{proof}

The relevance of $(A)$ comes from the following consequence of the minimum principle first discussed in \cite{KL1}.

\begin{prop}(Uniqueness of solutions on $\ell^p$)
 Assume $(A)$. Let $\alpha>0$, $p\in [1,\infty)$
  and $u\in \ell^p (V,m)$ with $(\widetilde{L} + \alpha) u \geq 0$ be
  given. Then, $u\geq 0$. In particular, any  $u\in \ell^p (V,m)$ with $(\widetilde{L} + \alpha) u =0$ satisfies $u\equiv 0$.
\begin{proof} We first show the first statement:
Assume the contrary. Then, there exists an $x_0\in V$ with
  $u(x_0)<0$.  By the previous minimum principle, $x_0$ is not a local minimum of $u$. Thus, there exists an $x_1$ connected to $x_0$ with $u(x_1)< u(x_0) < 0$.
 Continuing in this way we obtain a sequence $(x_n)$ of connected
 points with $u(x_n) < u(x_0)<0$. Combining this with $(A)$ we obtain a
 contradiction to $u\in \ell^p (V,m)$.

As for the 'In particular' part we note that both  $u$ and $-u$ satisfy the assumptions of the first statement. Thus, $u\equiv 0$.
\end{proof}
\end{prop}

\noindent\textbf{Remark.} The situation for $p=\infty$ is substantially more complicated as can be seen  by  our discussion of stochastic completeness in Section~\ref{Stochastic} and in particular part (ii) of Theorem \ref{main0}.
\medskip

Using the previous minimum principle it is not hard to prove the following result.  The result is in fact true for general Dirichlet forms as can be inferred from  \cite{Sto,SV}. For $U\subseteq V$ we denote by $Q_U$ the closure of the $Q$ restricted to $C_c (U)$ and by $L_U$ the associated operator.

\begin{prop} (Domain monotonicity)
 Let $(V,b,c)$  be a symmetric graph.  Let $K_1 \subseteq V$  be finite and  $K_2 \subseteq  V$ with  $K_1 \subseteq K_2$ be given. Then, for any $x\in K_1$
$$ (L_{K_1} + \alpha)^{-1} f(x) \leq ( L_{K_2} + \alpha)^{-1} f
(x)$$
for all $f\in \ell^2 (V,m)$ with $f\geq 0$ and $\supp f \subseteq K_1$. A similar statement holds for the semigroups.
\end{prop}
\medskip

\begin{prop}(Convergence of resolvents/semigroups)
 Let $(V,b,c)$  be a symmetric graph, $m$ a measure on $V$ with full support  and $Q$ the associated regular Dirichlet
 form. Let   $(K_n)$ be an
  increasing  sequence of finite  subsets
  of $V$ with $V =\bigcup K_n$.   Then,  $  (L_{K_n} + \alpha)^{-1}  f \to (L +\alpha)^{-1} f$, $ n\to \infty$ for any $f\in \ell^2 (K_1,m_{K_1})$.
(Here, $  (L_{K_n} + \alpha)^{-1}  f $ is
extended by zero to all of $V$.)
The corresponding statement also holds for
the semigroups.
\end{prop}

\section{Boundedness of the Laplacian} \label{Boundedness}
Our main topic  in this paper are the consequences of unboundedness of the Laplacian. In order to understand this unboundedness it is desirable to characterize  boundedness of this operator. This is discussed in this section. We start with a little trick on how to get rid of the $c$ in certain situations.
\medskip

Let $\dot V$ be the union of $V$ and a point at infinity $\infty$. We extend a function on $V$ to $\dot V$ by zero and let $b (\infty,x) = b(x,\infty)=c(x)$ for all $x\in V$.   We then have
$$ \sum_{y\in \dot V} b(x,y) =\sum_{y\in V} b(x,y) + c(x)$$
for all $x\in V$ and
$$Q(u)=\frac{1}{2}\sum_{x,y\in \dot V} b(x,y)(u(x) - u(y))^2$$
for all functions $u$ in $D(Q)$.

We define an averaged vertex degree $d=d_{b,c,m}$ by
$$d(x):=\frac{1}{m(x)}\ab{\sum_{y\in V}b(x,y)+c(x)}.$$
Note that $d(x)=n(x)/m(x)$, where $n$ was defined at the end of Section~1.3.
\medskip

\begin{theorem} \label{boundednesthm} Let $(V,b,c)$ be a weighted graph and $m: V\longrightarrow (0,\infty)$ a measure on $V$ and $\widetilde{L}$ the associated formal operator.  Then, the following assertions are equivalent:
\begin{itemize}
\item[(i)] There exists a $C\geq 0$ with $d(x) \leq C$ for all $x\in V$.
\item[(ii)] The form $Q$ is bounded on $\ell^2 (V,m)$.
\item[(iii)] The restriction of $\widetilde{L}$ to $\ell^2 (V,m)$ is bounded.
\item[(iv)] The restriction of $\widetilde{L}$ to $\ell^\infty (V)$ is bounded.
\end{itemize}
In this case the restriction of $\widetilde{L}$ to $\ell^p (V,m)$ is a bounded operator for all $p\in [1,\infty]$ and a bound is given by $ 2 C$ with $C$  from (i).
\begin{proof} By the considerations at the beginning of the section we can assume $c\equiv 0$. For $x\in V$ we let $\delta_x$ be the function on $V$ which is zero everywhere except in $x$, where it takes the value $1$.
\medskip

The equivalence between (ii) and (iii) is  obvious as  the operator associated to $Q$ is a densely defined  restriction of $\widetilde{L}$.
\medskip

Obviously (i) implies (iv) (with the bound $2 C$).  The implication (iv)$\Longrightarrow$ (i) follows by considering the vectors $\delta_x$, $x\in V$.
\medskip

(i) $\Longrightarrow $ (ii): As $(a-b)^2 \leq 2 a^2 + 2 b^2$ we obtain
\begin{eqnarray*}
Q(u,u) &=&  \frac{1}{2} \sum_{x,y\in V} b(x,y) (u(x) - u(y))^2\\
&\leq & \sum_{x,y\in V} b(x,y) u(x)^2 + \sum_{x,y\in V} b(x,y) u(y)^2\\
&\leq & C \sum_{x\in V} m(x) u(x)^2 +  C \sum_{y\in V} m(y) u(y)^2\\
&=& 2 C  \|u\|^2.
\end{eqnarray*}
Here, we used the symmetry of $b$ and the bound (i) in the previous to the  last step.
\medskip

(ii) $\Longrightarrow $ (i): This follows easily as $Q (\delta_x,\delta_x) = \sum_{y\in V} b(x,y)$ for all $x\in V$.
\medskip

It remains to show the last statement: By interpolation between $\ell^2$ and $\ell^\infty$, we obtain boundedness of the operators on $\ell^p (V,m)$ for $p\in [2,\infty]$. Using symmetry we obtain the boundedness for $p\in  [1,2)$. Alternatively, we can directly establish that (i) implies the boundedness of the restriction of $\widetilde{L}$ on $\ell^1 (V,m)$. As a bound for the operator norm on $\ell^\infty$ and on $\ell^2$ is $2 C$, we obtain this same bound on all $\ell^p$.
\end{proof}
\end{theorem}

\noindent\textbf{Remark.} The theorem can be seen as a generalization of the well known fact that a stochastic matrix generates an operator which is bounded on all $\ell^p$.
\medskip

Note that the theorem gives in particular that boundedness of the operator $\widetilde{L}$  on $\ell^2 (V,m)$ is equivalent to boundedness on $\ell^\infty (V)$. This is far from being true for all symmetric operators on $\ell^2 (V,m)$. For example, let $A$ be the operator on $\ell^2 (\NN,1)$ with matrix given by $a_{x,y} = 1/x$ if $y=1$ and $a_{x,y} = 1/y$ if $x=1$ and $a_{x,y} =0$ otherwise. Then, $A$ is bounded on $\ell^2$ but not on $\ell^\infty$. Conversely, using e.g. the measure $m(x) = x^{-4}$ on $\NN$  and suitable operators with  only one or two ones in each row it is not hard to construct a bounded operator on $\ell^\infty (\NN)$ which is symmetric but  not bounded on $\ell^2 (V,m)$. Of course, if $m$ is such that $\ell^2 (V,m)$ is contained in $\ell^\infty (V)$ then any bounded operator on $\ell^\infty$ which is symmetric (and hence closed) on $\ell^2 $ must be bounded as well.

\section{Co-area formulae}\label{Co}
In this section we discuss some co-area type formulae. These formulae are well known for locally finite graphs e.g. \cite{CGY} and carry over easily to our setting. They are useful in many contexts as e.g. the estimation of eigenvalues via isoperimetric inequalities.  We use them in this spirit as well.
\medskip

We start with some notation. Let $(V,b,c)$ be a weighted graph with $c\equiv 0$, (which can assume without loss of generality by the trick mentioned in the beginning of Section~\ref{Boundedness}). For a subset $\varOmega\subseteq V$ we define
$$\partial \varOmega :=\{(x,y) : \{x,y\} \cap \varOmega\neq \emptyset\;\:\mbox{and}\;\: \{x,y\} \cap V\setminus \varOmega\neq \emptyset\} $$
and
$$|\partial \varOmega|:=\frac{1}{2} \sum_{(x,y)\in \partial\varOmega} b(x,y).$$

We can now come to the so called co-area formula.

\begin{theorem} (Co-area formula) \label{coareaeins}  Let $(V,b,c)$ be a weighted graph with $c\equiv 0$. Let $ f : V\longrightarrow \R$ be given and define for $t\in \R$ the set $\varOmega_t:=\{x\in V: f(x) >~t\}$. Then,
$$\frac{1}{2} \sum_{x,y\in V} b(x,y) |f(x) -  f(y)| = \int_0^\infty |\partial \varOmega_t| dt. $$
\begin{proof} For $x,y\in V$ with $x\neq y$ we define the interval $I_{x,y}$ by
$$I_{x,y}:=[\min\{f(x),f(y)\}, \max\{f(x),f(y)\})$$
and let $|I_{x,y}|$ be the length of the interval. Let $1_{x,y}$ be the characteristic function of $I_{x,y}$. Then, $(x,y)\in \partial\varOmega_t$  if and only if  $t\in I_{x,y}$. Thus,
$$ |\partial \varOmega_t| = \frac{1}{2} \sum_{x,y\in V} b(x,y) 1_{x,y} (t).$$
Thus, we can calculate
\begin{eqnarray*}
\int_0^\infty |\partial \varOmega_t| dt &=& \frac{1}{2} \int_0^\infty \sum_{x,y\in V} b(x,y) 1_{x,y} (t) dt\\
&=& \frac{1}{2} \sum_{x,y\in V}b(x,y) \int_0^\infty 1_{x,y} (t) dt\\
&=& \frac{1}{2} \sum_{x,y\in V} b(x,y) |f(y) - f(x)|.
\end{eqnarray*}
This finishes the proof.
\end{proof}
\noindent\textbf{Remark. }  Note that the proof is essentially a Fubini type argument.
\end{theorem}

The preceding formula can be seen as a first order co-area formula as it deals with differences of functions. There is also a zeroth order co-area type formula dealing with functions themselves. This is discussed next.

\begin{theorem}\label{coareazwei}Let $V$ be a countable set and $m:V\longrightarrow (0,\infty)$ a measure on $V$.  Let $ f : V\longrightarrow [0,\infty)$ be given and define for $t\in \R$ the set $\varOmega_t:=\{x\in V: f(x) > t\}$. Then,
$$\sum_{x\in V} m(x) f(x) = \int_0^\infty m (\varOmega_t) dt.$$
\begin{proof} We have $x\in \varOmega_t$ if and only if $1_{(t,\infty)} (f(x)) =1$. Thus, we can calculate
\begin{eqnarray*}
\int_0^\infty m(\varOmega_t) dt & = & \int_0^\infty \sum_{x\in \varOmega_t} m(x) dt\\
&=& \int_0^\infty \sum_{x\in V} m(x) 1_{(t,\infty)} (f(x)) dt\\
&=& \sum_{x\in V} m(x) \int_0^\infty 1_{(t,\infty)} (f(x)) dt\\
&=& \sum_{x\in V} m(x)  f(x).
\end{eqnarray*}
This finishes the proof.
\end{proof}
\end{theorem}

\section{Isoperimetric inequalities and lower bounds on the (essential) spectrum}\label{Isoperimetric}
In this section we will provide lower bound on the infimum of the (essential) spectrum using an isoperimetric inequality. This will allow us in particular to provide criteria for emptiness of the essential spectrum. Our considerations  extend the corresponding parts of \cite{Dod0,DK,Fuj,Kel} (as discussed in more detail below).
\medskip

We start with some notation used throughout this section.
Let a weighted graph $(V,b,c)$ with a measure $m: V\longrightarrow (0,\infty)$ and the associated Dirichlet form $Q$ be given.  In this setting we define   the constant $\al(U)=\al_{b,c,m}(U)$ for a subset $U\subseteq V$ by
$$\al(U)=\inf_{W\subseteq U, |W|<\infty} \frac{|\dd W|}{m(W)},$$
where as introduced in the previous section
$$|\dd W|=\sum_{x\in W, y\not\in W}b(x,y)+\sum_{x\in W}c(x).$$
Note that for a finite set $W$ and  the  characteristic function $1_W$  of $W$ one has
\begin{equation}\label{var}
\frac{|\dd W|}{m(W)}= \frac{Q(1_W)}{\aV{1_W}^2}.
\end{equation}
Recall the definition of the normalizing measure $n$ on $V$
$$n(x)=\sum_{y\in V}b(x,y)+c(x).$$
Thus, we have two measures and thus two Hilbert spaces at our disposal.  To avoid confusion, we will write $\|\cdot\|_m$ and $\|\cdot\|_n$ for the corresponding norms whenever necessary.
\medskip

Note that $d(x)=n(x)/m(x)$. Define maximal and minimal averaged vertex degree by
$$d_U=d_{b,c,m}(U)=\inf_{x\in U}d(x)$$
and
$$ D_U=D_{b,c,m}(U)=\sup_{x\in U}d(x),$$
where $d$ is the averaged vertex degree, which was defined in Section~~\ref{Boundedness}
Recall $d(x)=n(x)/m(x)$ for $x\in V$.
\medskip

We will also need the restrictions of operators on $V$ to subsets of $V$. As in the end of Section~2
denote the closure of the restriction of a closed semibounded form $Q$ with domain containing $C_c (V)$ to $C_c (U)$ by $Q_U$ and its associated operator by $L_U$ (for $U\subseteq V$ arbitrary).
\medskip

For later use we also note that for the  Dirichlet form $Q$  associated to a graph $(V,b,c)$ with measure $m$ on $V$
we have
$$ \inf \sigma (L_U) =\inf_{u\in C_c (U)}  \frac{ Q(u)}{\|u\|^2} \leq \alpha (U) \leq  \inf_{x\in U} d (x) = d_U$$
for any $U\subseteq V$.
Here, the first equality is just the variational principle for forms, the second step follows from the definition of $\alpha$ and the last estimate follows by choosing $W = \{x\}$ for $x\in U$.
In particular, $\alpha$ gives upper bound on the infimum of the spectrum. It is a  remarkable (and well known) fact  that $\alpha>0$ implies also a  lower bounds on the infimum of spectra. This is the core of the present section.

\subsection{An isoperimetric inequality}
In this subsection we provide an isoperimetric inequality in our setting. This inequality (and its proof) are   generalizations of the corresponding considerations of \cite{DK,Fuj,Kel} to our  setting.

\begin{prop}\label{p:Q}  Let $(V,b,c)$ be a weighted graph, $m: V\longrightarrow (0,\infty)$ a measure on $V$ and $Q$ the associated regular Dirichlet form.
Let $U\subseteq V$ and $\phi \in C_c (U)$. Then
\begin{eqnarray*}
Q(\ph)^2-2\|\varphi\|_{n}^2 Q(\ph)+\al_{b,c,m}(U)^2\|\varphi\|_{m}^4\leq 0.
\end{eqnarray*}
\begin{proof}  By the trick introduced at the beginning of Section
we can assume
without loss of generality  that $c\equiv 0$.
Define now $\A$ by
$$\A=\frac{1}{2}\sum_{x,y\in \dot V}b(x,y)\av{\varphi(x)^2-\varphi(y)^2} =
\sum_{x,y\in \dot V}b(x,y)|\varphi(x)-\varphi(y)| |\varphi(x)+\varphi(y)|
.$$
Following ideas  of  \cite{DK}  for locally finite graphs (see \cite{Fuj,Kel} as well) we  now proceed as follows:
By Cauchy-Schwarz inequality and  a direct computation we have
\begin{eqnarray*}
\A^2  \leq
Q(\ph)
\ab{\frac{1}{2}\sum_{x,y\in \dot V}b(x,y)\av{\varphi(x)+\varphi(y)}^2}=Q(\ph)\ab{2\aV{\varphi}_{n}^2-Q(\ph)}.
\end{eqnarray*}
On the other hand we can use the first  co-area formula (with $f= \varphi^2$), the definition of $\alpha$ and the second co-area formula to estimate
\begin{eqnarray*}
\A =\int_0^\infty |\partial \varOmega_t| dt
\geq  \alpha \int_0^\infty m(\varOmega_t) dt
= \alpha\sum_{x\in V} m(x) \varphi^2 (x)
= \alpha \|\varphi\|_m^2.
\end{eqnarray*}
Combining the two estimates on $\A$ we obtain
$$   Q(\ph)\ab{2\aV{\ph}_{n}^2-Q(\ph)} \geq     \|\varphi\|_m^4.$$
This yields the desired result.
\end{proof}
\end{prop}

\subsection{Lower bounds for the infimum of the spectrum}
In this section we use the isoperimetric inequality of the previous section to derive bounds on the form $Q$. This is in the spirit of \cite{DK,Fuj,Kel}.  As usual we write
$$ a \leq Q \leq b$$
(for $a,b\in \R$) whenever
$$ a \|u\|^2 \leq Q(u) \leq b \|u\|^2$$
 for all $u\in D(Q)$.

\begin{prop}\label{p:sp}
Let $(V,b,c)$ be a weighted graph, $m: V\longrightarrow (0,\infty)$ a measure on $V$ and $Q$ the associated regular Dirichlet form. Let $U\subseteq V$ be given and $Q_U$ the restriction of $Q$ to $U$. Then,
$$d_U\ab{1-\sqrt{1-\al_{b,c,n}(U)^2}}
\leq  Q_U \leq
D_U\ab{1 + \sqrt{1 - \al_{b,c,n}(U)^2}}.$$
If $D_U<\infty$ then furthermore
\begin{eqnarray*}
D_U-\sqrt{D_U^2-\al_{b,c,m}(U)^2}\leq Q_U \leq &D_U+\sqrt{D_U^2-\al_{b,c,m}(U)^2}.
\end{eqnarray*}
\begin{proof} We start by proving the first statement.  Consider an arbitrary $\varphi\in C_c (U)$ with $\|\varphi\|_n=1$. Then, Proposition \ref{p:Q} (applied with  $m =n$) gives
$$ Q(\varphi)^2-2  Q(\varphi)+\al_{b,c,n}(U)^2\leq 0 $$
and hence
$$ 1- \sqrt{1 - \al_{b,c,n}(U)^2}  \leq Q(\varphi) \leq 1 + \sqrt{1 - \al_{b,c,n}(U)^2}. $$
As this holds for all $\varphi \in C_c (U)$ with $\|\varphi\|_n=1$ and
$$ d_U \|\varphi\|_m \leq \|\varphi\|_n \leq D_U \|\varphi\|_m
$$ by definition of $d_U$ and $D_U$, we obtain the first statement.
\medskip

We now turn to the last statement. By definition of $D_U$ we have $\|\varphi\|_n \leq D_U\|\varphi\|_m$. Thus, Proposition \ref{p:Q} gives
$$ Q(\ph)^2-2 D_U \|\varphi\|_m^2 Q(\ph)+\al_{b,c,m}(U)^2\|\varphi\|_m^4\leq 0.$$
Considering now $\varphi\in C_c (U)$ with $\|\varphi\|_m=1$ we find  that
$$D_U-\sqrt{D_U^2-\al_{b,c,m}(U)}\leq Q(\ph)\leq D_U+\sqrt{D_U^2-\al_{b,c,m}(U)}$$
for all such $\varphi$.
This finishes the proof.
\end{proof}
\end{prop}

As a first  consequence of the previous proposition we obtain the following corollary first proven for $m=n$, and locally finite graphs  in \cite{Fuj}.

\begin{coro}  For a weighted graph $(V,b,c)$ and $m=n$ we obtain
$$1-\sqrt{1-\al_{b,c,n}^2}\leq Q \leq 1+\sqrt{1-\al_{b,c,n}^2}. $$
\end{coro}

A second consequence of the above proposition is that the bottom of the spectrum being zero can be characterized by the constant $\al$ in the case of bounded operators. This is our version of the well known result that  a graph with finite vertex degree is amenable if and only if  zero belongs to the spectrum of the corresponding Laplacian.

\begin{coro}  Let $(V,b,c)$ be a weighted graph and $D_U<\infty$ for $U\subseteq V$. Then $\inf \si(L_U)=0$ if and only if $\al_{b,c,m}(U)=0$.
\begin{proof} The direction '$\Longrightarrow$' follows from Proposition \ref{p:sp} and the other direction '$\Longleftarrow$' follows  directly from equation \ref{var}.
\end{proof}
\end{coro}

\noindent\textbf{Remark.} The direction '$\Longleftarrow$' in the previous corollary does not depend on the assumption $D_U<\infty$ for $U\subseteq V$ and is true in general.

\subsection{Absence of essential spectrum}
In this subsection we use the results of the previous subsection to study absence of essential spectrum. The key idea is that the essential spectrum of an operator is a suitable limit of the spectra of restrictions 'going to infinity'. This reduces the problem of proving absence of essential spectrum to proving lower bounds on the spectrum 'at infinity'. For
unweighted graphs this has been done in \cite{Fuj,Kel}.
\medskip

Let $(V,b,c)$ be a weighted graph.  Let $\mathcal K$ be the set of finite sets in $V$.  This set is directed with respect to inclusion and hence a net. Limits along this net will be denoted by $\lim_{K\in \mathcal K}$ and we will say that $K$ tends to $V$.
We then define
$$\al_{b,c,m}(\dd V)=\lim_{K\in\mathcal K} \al_{b,c,m}(V\setminus K).$$
Likewise let
\begin{eqnarray*}
  d_{\dd V} &=& d_{b,c,m}(\dd V)=\lim_{K\in\mathcal K}d_{b,c,m}(V\setminus K), \\
   D_{\dd V}&=&  D_{b,c,m}(\dd V)=\lim_{K\in\mathcal K}D_{b,c,m}(V\setminus K).
\end{eqnarray*}

The following proposition is certainly well known and has in fact  already been used in the past (see e.g. \cite{Kel}).  We include a proof as we could not find one in the literature.  Note also that our result is more general than the result mentioned e.g. in \cite{Kel}  as we deal with forms. Note that the compactness assumption is fulfilled if we consider operators on locally finite graphs.

\begin{prop}\label{p:B} Let $Q$ be a closed form on $\ell^2 (V,m)$, whose domain of definition contains $C_c (V)$.  Let $Q$ be bounded below.  Then,
\begin{eqnarray*}
\inf\se(B)=\lim_{K\in \mathcal K} \inf\si(B_{V\setminus K}).
\end{eqnarray*}
and if $Q$ is bounded above then
$$\sup \se (B) =\lim_{K\in \mathcal K} \sup \si (B_{V\setminus K})$$
holds, whenever the operator $B$  associated to $Q$ and the operator $B_{V\setminus K}$ associated to $Q_{V\setminus K}$ for finite $K\subseteq V$ are compact perturbations of each other.
\begin{proof} It suffices to show the statement for $Q$ which are bounded below (as the other statement then follows after replacing $Q$ by $-Q$).

Without loss of generality we can assume $Q\geq 0$. Let $\lm_0:=\inf\se(B)$.

As the essential spectrum does not change by compact perturbations we have
$\se(B)=\se(B_{V\setminus K})\subseteq \sigma(B_{V\setminus K})$ and hence
$$\lm_0\in \sigma(B_{V\setminus K})$$
for any finite $K\subseteq V$. This gives
$$\inf \se (B)\geq \lim_{K\in \mathcal K} \inf  \si (B_{V\setminus K}).$$

To show the opposite inequality it suffices to  prove that for arbitrary $\lambda < \lm_0$ we have
$\inf \sigma (B_{V\setminus K}) > \lambda$
for  all sufficiently large finite $K$.  Fix $\lm_1$ with
$$\lm<\lm_1<\lm_0$$
and choose $\de>0$ such that
$\lm+\de<\lm_1$. Moreover let
$$\varepsilon=\frac{\lm_1-(\lm+\de)}{\lm_1+1}.$$
The spectral projection $E_{(-\infty,\lm_1]}$ of $B$ to the interval $(-\infty,\lm_1]$  is a finite rank
operator since $B\geq 0$.  This easily implies
$$ \lim_{K\in\mathcal K} \| E_{(-\infty,\lm_1]} P_K\| =0,$$
where $P_K$ is the  projection onto $\ell^2 (V\setminus K, m)$.
Thus, there is $K_\varepsilon$ finite with
$$
\| E_{(-\infty,\lm_1]} P_K\|^2 \leq \varepsilon
$$
for all $K\supseteq K_\varepsilon$ finite.   In particular, we have
\begin{equation}\label{e:E}
\| E_{(-\infty,\lm_1]} \psi  \|^2 \leq \varepsilon
\end{equation}
for all $\psi \in \ell^2 (V\setminus K_\varepsilon,m)$ with $\|\psi\|=1$ (as for such $\psi$ we have $\psi = P_{K_\varepsilon} \psi$).

Consider now a finite $K$ with $K\supseteq K_\varepsilon$ and   let  $\psi\in \ell^2(V\setminus K ,m)$ be given  with $\|\psi\|=1$  such that
$$Q (\psi) = Q_{V\setminus K}  (\psi) \leq
(\inf\sigma(B_{V\setminus K})+\varepsilon).$$
Let  $\rho_\psi(\cdot)$  be the  spectral measure associated to $B$ and $\psi$.  Then
\begin{eqnarray*}
Q (\psi) &=& \int_0^\infty t d \rho_\psi (t)\\
&\geq& \int_{\lm_1}^\infty t \; d\rho_\psi(t)\\
&\geq& \lm_1\int_{\lm_1}^\infty \; d\rho_\psi(t)\\
&=&\lm_1(\bs{\psi,\psi} - \bs{E_{(-\infty,\lm_1]}\psi,E_{(-\infty,\lm_1]}\psi} )\\
&\geq & \lm_1 (1 - \varepsilon).
\end{eqnarray*}
In the first step we used that $B$ is positive and in  the  last
step we used  \eqref{e:E}.  By our choice of $\psi$ and $\varepsilon$ we get
$$\inf\sigma(B_{V\setminus K} )\geq Q (\psi) -\varepsilon\geq
\lm_1(1-\varepsilon)-\varepsilon=\lm+\de>\lm.$$
This finishes the proof.
\end{proof}
\end{prop}

Combining this proposition with Proposition \ref{p:sp} one gets estimates for the essential spectrum of the operator $L$.
\medskip

The following provides  a generalization of a main result of  Fujiwara's theorem \cite{Fuj} to our setting. Fujiwara's result  deals with
$m=n$.

\begin{theorem}  Let $(V,b,c)$ be a locally finite weighted graph, $m: V\longrightarrow (0,\infty)$ a measure on $V$ and $Q$ the associated regular Dirichlet form. Assume
 $D_{\dd V}=D_{b,c,m}(\dd V)<\infty$. Then, $\se(L)=\{D_{\dd V}\}$ if and only if
$\al_{b,c,m}(\dd V)=D_{\dd V}$.
\begin{proof} One direction '$\Longleftarrow$' follows directly from Proposition \ref{p:sp} and Proposition \ref{p:B}.
The other direction '$\Longrightarrow$' follows from
$$\inf\si(L^{}_U)\leq \al_{b,c,m}(U)\leq D_{b,c,m}(U)$$
for $U\subseteq V$  and Proposition \ref{p:B} by taking $U = V\setminus K$ for $K$ finite and considering the limit for $K$ tending to $V$.
\end{proof}
\end{theorem}

\noindent\textbf{Remark.}  The assumption $D_{b,c,m}(\dd V)<\infty$ implies boundedness of the operator (see Section~\ref{Boundedness}). Thus, $\se(L)$ must be non-empty in this case. Proposition \ref{p:B} shows that $\inf \sigma (L_{V\setminus K})$ and $\sup \sigma (L_{V\setminus K})$ converge necessarily to points in the essential spectrum of $L$ (for $K$ tending to $V$). The only way how the essential spectrum can consist of only one point is then that both limits agree.
As $\inf \sigma (L_{V\setminus K}) \leq \alpha (V\setminus K)$ and $\sup \sigma (L_{V\setminus K}) \geq D_{b,c,m}$  this is only possible for $\al_{b,c,m}(\dd V)=D_{\dd V}$.  In this way the theorem characterizes the only way how essential spectrum can consist of only one point.
\medskip

The next theorem is a generalization to our setting  of Theorem 2 in \cite{Kel}, which deals with locally finite graphs and $m\equiv 1$.

\begin{theorem} \label{absenceess} Let $(V,b,c)$ be a locally finite weighted graph, $m: V\longrightarrow (0,\infty)$ a measure on $V$ and $Q$ the associated regular Dirichlet form. Assume  $\al_{b,c,n}>0$. Then $\se(L)=\emptyset$ if and only if $d_{\dd V}=\infty$.
\begin{proof} One direction '$\Longleftarrow$' follows directly from Proposition \ref{p:sp} and \ref{p:B}.
The other direction '$\Longrightarrow$'  follows from the fact that for all $U\subseteq V$ we have
$\inf\si(L_U)\leq d_{b,c,m}(U)$ and Proposition \ref{p:B}.
\end{proof}
\end{theorem}

\section{An application}
In this section we consider a locally finite graph i.e.,  $(V,b,0)$ with $b$ taking values in $\{0,1\}$ with the measure  $m\equiv 1$. Let  $Q_0$ be the associated form and $ \Delta$ the associated operator. Let $c:V\longrightarrow [0,\infty)$ be given and define $L$ to be the operator associated to $Q_{b,c,m}$. Thus,
$$ L = \Delta + c$$
(at least on the formal level). This decomposition of $L$ leads to a similar decomposition of the parameters $\alpha$. In this way,  both  the geometry (encoded by $b$) and the potential (encoded by $c$) can lead to absence of essential spectrum according to the preceding considerations.  This is discussed in further details next.
\medskip

The Cheeger constant $\beta_U$ of a subset $U\subseteq V$  is the smallest number such that for all finite $W\subseteq U$
$$|\partial W|\geq\beta_U\vol (W),$$
where $|\partial  W|=\langle \Delta 1_W,1_W\rangle= \sum_{x\in W,y\notin W} b(x,y)$ is defined as above and $\vol (W)=\aV{1_W}_{n}^2 = \sum_{x\in W}  n(x) $.
If $\be_V>0$ one says that the graph is hyperbolic.
Furthermore, let $\gm_U$ be given as the smallest number such that for all finite $W\subseteq U$
$$c(W)\geq \gm_U \vol(W),$$
where $c (W)=\langle c 1_W,1_W\rangle=\sum_{x\in W}c(x)$.

For example $\gm_V>0$, if there is $C>0$ such that $c(x)\geq C d(x)$, where $d(x)$ is the vertex degree.

Finally let
$$\be_{\dd V}=\lim_{K\in\mathcal{K}}\be_{V\setminus K}\qand \gm_{\dd V}=\lim_{V\in\mathcal{K}}\gm_{V\setminus K}.$$

Hence the preceding section immediately gives the  following corollary of Theorem \ref{absenceess}.

\begin{coro}Let $\be_{\dd V}>0$ or $\gm_{\dd V}>0$. Then $\se(H)=\emptyset$ if and only if $d(x_n)+c(x_n)\ra\infty$ along any infinite sequence $(x_n)$ of vertices which eventually leaves every compact set.
\end{coro}

\section{Graph Laplacians and Markov processes}\label{Graphmarkov}
We have  already discussed that our Laplacians come from Dirichlet forms.
Now,  Dirichlet forms and symmetric Markov processes are intimately connected.   The crucial link is given by the semigroup generated by a Dirichlet form. The connection to Markov processes means that
\begin{itemize}
\item there is a wealth of results on the semigroup associated to a graph Laplacian,
\item there is a good interpretation of properties of the semigroup in terms of a stochastic process.
\end{itemize}
Details are discussed in this section.

\subsection{Graph Laplacians, their  semigroup and the  heat equation}
Let  a measure $m$ on $V$ with full support and a graph $(b,c)$
over $V$ be given. Let $Q$ be the associated form and $L$ its generator.
\medskip

Standard theory \cite{Dav3,Fuk,MR} implies that the operators of the associated semigroup $e^{-t
  L}$, $t\geq 0$, and the associated  resolvent $\alpha (L +\alpha)^{-1}$, $\alpha >0$
are  positivity preserving  and even markovian. Positivity preserving  means that they  map nonnegative functions to nonnegative functions.  Markovian  means that they map nonnegative
functions bounded by one to nonnegative functions bounded by one.
\medskip

This can be used to show that  semigroup and resolvent  extend to  all $\ell^p (V,m)$, $1\leq p\leq \infty$.  These  extensions  are consistent i.e., two of them agree on their common domain \cite{Dav1}.
The corresponding generators are denoted  by $L_p$, in particular $L
= L_2$.
We can  describe the action of  the operator $L_p$ explicitly.  More precisely, the situation on $\ell^2$ (see Proposition \ref{generator} and Theorem \ref{generatorthm}) holds here as well:

\begin{theorem}\label{generator-p} Let $(V,b,c)$ be a weighted graph and $m$ a
 measure on $V$ of  full support. Then, the operator $L_p$ is a restriction of $\widetilde{L}$ for any $p\in [1,\infty]$. If furthermore   $(A)$ holds, then the operator  $L$ is the
  restriction of $\widetilde{L}$ to
  $$  \{ u\in \ell^p (V,m) : \widetilde{L} u \in \ell^p
  (V,m)\}. $$
\end{theorem}

A function $N : [0,\infty)\times V\longrightarrow \RR$ is called a solution of
the heat equation if for each $x\in V$ the function $t\mapsto N_t (x)$ is
continuous on $[0,\infty)$ and differentiable on $(0,\infty)$ and for each
$t>0$ the function $N_t$ belongs to the domain of $\widetilde{L}$, i.e., the vector space $\widetilde{F}$ and the
equality
$$ \frac{d}{dt} N_t (x) =  - \widetilde{L} N_t (x)$$
holds for all $t>0$ and $x\in V$.  For a bounded solution $N$ validity of this equation can  easily be seen to automatically  extend to $t=0$ i.e., $t\mapsto N_t (x)$ is
differentiable on $[0,\infty)$ and
$ \frac{d}{dt} N_t (x) =  - \widetilde{L} N_t (x)$ holds for any $t\geq0$.
\medskip

The following theorem is a standard result in the theory of semigroups. A proof in our context can be found in \cite{KL1} (see \cite{Woj1,Woj2,Web} for related material on special graphs).

\begin{theorem}\label{solution}
 Let $L$ be a selfadjoint restriction of $\widetilde{L}$, which is
 the generator of a Dirichlet form on $\ell^2 (V,m)$.
 Let $v$ be a bounded function on $V$ and define  $N: [0,\infty)\times
  V\longrightarrow \RR$ by $N_t (x) := e^{-t L} v (x)$. Then, the function $N(x) : [0,\infty)\longrightarrow \RR$, $t\mapsto
  N_t (x)$, is differentiable and satisfies
$$ \frac{d}{d t} N_t (x) = - \widetilde{L} N_t (x)$$
for all $x\in V$ and $t\geq 0$.
\end{theorem}

Let us conclude this section by noting that the semigroups are positivity improving for connected graphs. This has been shown in \cite{KL1} in our setting after earlier results in \cite{Dav3,Web,Woj1} for locally finite graphs.

\begin{theorem}\label{positivityimproving} (Positivity improving)  Let $(V,b,c)$ be a  connected graph and $L$ be the associated operator. Then, both the semigroup $e^{-t L}$, $t>0$, and the resolvent $(L + \alpha)^{-1}$, $\alpha>0$, are positivity improving (i.e., they map nonnegative nontrivial $\ell^2$-functions to strictly positive functions).
\end{theorem}

\subsection{Connection to Markov processes} In this section we discuss the relationship between Dirichlet forms and Markov processes in our context. Let $Q$ be the Dirichlet form associated to a weighted graph $(V,b,c)$ with measure $m$. For convenience we assume $m\equiv 1$. Let $L$ be the associated operator and  $e^{-t L}$, $t>0$, the associated semigroup. We will take the point of view that we already know  that  $e^{-t L}$ is a semigroup of transition properties of a  Markov process. We will then show how we can identify the key quantities of the Markov process in terms of the graph $(V,b,c)$.
\medskip

A (time homogenous) Markov process on $V$ consists of a particle moving in time without memory between the points of $V$. It is characterized by two sets of quantities: These are

\begin{itemize}
\item a function $a : V\longrightarrow [0,\infty)$ such that $e^{-t a_x}$ is the probability that a particle in  $x$ at time $0$ is still in $x$ at time $t$.

\item a function $ q: V\times V\longrightarrow [0,\infty)$ such that $q_x (y)$ is the probability that the particle jumps to $y$ from $x$.
\end{itemize}
Given such a Markov process we can define
\begin{align*}
P_t (x,y) :=&\mbox{Probability that the particle is in $y$ at time $t$}\\
&\mbox{if it starts in $x$ at time $0$}
\end{align*}

for $t\geq 0$, $x,y\in V$ and the operators $P_t$ provide a semigroup of operators.  It is then possible to infer the quantities  $a$ and $q$ from the  behavior of $P_t$ for small $t$ in the following way:

\noindent $P_t (x,x)$ is the probability to find the particle at $x$ at time $t$  (for a particle starting at $x$ at time $0$). This means that the particle has either stayed at $x$ for the whole time between $0$ and $t$  or has jumped from $x$ away and come  back by the time $t$. The probability that the particle stayed in $x$ (i.e., did not move away) is $e^{- t a_x}$. The event that the particle left $x$ and returned by the time $t$  means that the particle left $x$, which occurs with probability $1 - e^{- t a_x}$,  and then returned from $V\setminus\{x\}$ to $x$ in the remaining time, which occurs with probability  $r (t)$ going   to zero for $t\to 0$.  Accordingly we have
$$P_t (x,x) = e^{-t a_x} + \phi_x (t),$$
where $\phi_x$ summarizes  the probability of returning to $x$, is therefore bounded by $(1-e^{- t a_x}) r(t)$  and hence has derivative equal to zero at $t=0$. We  therefore obtain
$$ \left.\frac{d}{dt}\right |_{t=0} P_t (x,x) = - a_x + \phi_x' (0)  = -a _x.$$
By a similar reasoning the probability $P_t (x,y)$ is governed by the event that the particle starts at $x$ at time $0$ and has done  one jump to $y$  and then stayed in $y$  up to the time $t$. The probability  $p_t$ for this event satisfies
$$ (1 -e^{- t a_x}) q_x (y)  e^{- t a_y } \leq p_t \leq (1 -e^{- t a_x}) q_x (y).$$
Here, the term $e^{- t a_y}$ serves to take into account that the particle did not leave $y$.
 Accordingly,
$$P_t (x,y) = p_t    + \psi(t),$$
where the derivative of $\psi$ at $0$ is zero and we obtain
$$ \left.\frac{d}{dt}\right |_{t=0} P_t (x,y) = a_x q_x (y)  + \psi ' (0)  = a_x q_x (y).$$

We now return to the Dirichlet form setting.  As $e^{-t L}$ describes  a Markov process  we can now set
$$ P_t (x,y) =  \langle e^{-t L} \delta_x,\delta_y\rangle $$
for $t\geq 0$, $x,y\in V$  and use  this to calculate the the $a$'s and $q$'s in terms of $b$ and $c$ as follows:
$$\sum_{y\in V} b(x,y) + c(x) = Q (\delta_x,\delta_x) = \left.\frac{d}{dt}\right |_{t=0} \langle e^{-t L} \delta_x,\delta_x\rangle = \left.\frac{d}{dt}\right |_{t=0}  P_t (x,x) = - a_x$$
and
$$ - b(x,y) = Q (\delta_x,\delta_y) =  \left.\frac{d}{dt}\right |_{t=0} \langle e^{-t L} \delta_x,\delta_y\rangle = \left.\frac{d}{dt}\right |_{t=0}  P_t (x,y)= q_x (y) a_x.$$
This gives
$$ q_x (y) = \frac{b(x,y)}{\sum_{z\in V} b(x,z) + c(x)},\quad  a_x = \sum_{z\in V} b(x,z) + c (x)$$
for all $x,y\in V$. Note that symmetry of $b$ does not imply symmetry of $q$ but rather
$$ a_x q_x (y) = a_y q_y (x).$$
If $m$ is not identically equal to one, we will have to normalize  the formula for $P$ above by setting
$$P_t (x,y) = \frac{1}{m(x) m(y)} \langle e^{- t L} \delta_x, \delta_y\rangle$$
and change the emerging formulae accordingly.

\section{Stochastic completeness}\label{Stochastic}
We consider a Dirichlet form $Q$ on a weighted graph $(V,b,c)$ with associated operator $L$ and semigroup $e^{-t L}$.  The preceding  considerations show that
$$ 0 \leq e^{-t L} 1 (x)\leq 1$$
for all $t\geq 0$ and $x\in V$. The question, whether the second inequality is actually an equality  has received quite some attention.  In the case of vanishing killing term, this is discussed under the name of stochastic completeness or conservativeness.
In fact, for $c\equiv 0$ and $b(x,y)\in \{0,1\}$ for all $x,y\in V$, there is a  characterization of stochastic
completeness  of Wojciechowski \cite{Woj1} (see the introduction for discussion of  related results  of Feller \cite{Fel} and Reuter \cite{Reu} as well). This
characterization is an analogue to corresponding results
on manifolds of Grigor'yan \cite{Gri} and results of  Sturm for general strongly
local Dirichlet forms \cite{Stu}.
\smallskip

Our first main result
concerns a version of this result for  arbitrary regular Dirichlet forms on graphs. As we allow for a killing term $c$  we have to replace $e^{-t L} 1 $ by the function
$$ M_t (x) := e^{-tL} 1 (x) + \int_0^t (e^{-s L} \frac{c}{m} )(x) ds, \quad x\in V. $$

It is possible (and necessary) to show that this quantity is well defined. In fact, it
can be proven that it satisfies $0\leq M \leq 1$ and that for each $x\in V$, the function $t\mapsto M_t (x)$ is continuous and even differentiable \cite{KL1}.
Note that for $c\equiv 0$,  $M = e^{-t L} 1$ whereas for $c\neq 0$ the inequality
$M_t >e^{- t L} 1$ holds  on any connected component of $V$ on which $c$ does
not vanish identically (as the semigroup is positivity improving).
\smallskip

We can give an interpretation of $M$ in terms of a diffusion
process on $V$ as follows: For $x\in V$, let $\delta_x$ be the characteristic function of $\{x\}$.
A diffusion
on $V$ starting in  $x$ with normalized measure is then given by $ \delta_x/{m(x)}$ at time $t=0$. It will yield to the amount of heat
$$\langle e^{-t L} \frac{\delta_x}{m(x)}, 1\rangle = \langle \frac{\delta_x}{m(x)}, e^{-t L} 1\rangle = \sum_{y\in V} e^{-tL}(x,y) = e^{-t L} 1 (x) $$
within $V$ at the time $t$.  Thus, the first term of $M$ describes the amount of heat within the graph at a given time.
\smallskip

Moreover, at each time $s$ the
rate of heat killed at the vertex $y$ by the  killing term $c$ is given by $  e^{-sL} (x,y)c(y)/m(y) $. The total amount of heat killed at $y$ till the time $t$ is then
given by $\int_0^t   e^{ -sL} (x,y)c(y)/m(y)  ds$. The total amount of
heat killed at all vertices by $c$ till the time $t$ is accordingly given by
$$ \sum_{y\in V}  \int_0^t   e^{- sL} (x,y)\frac{c(y)}{m(y)} ds = \int_0^t
\sum_{y\in V} e^{- sL} (x,y) \frac{c(y)}{m(y)} ds =  \int_0^t (e^{- sL} \frac{c}{m}) (x)
ds.  $$
Thus, the second term of $M$ describes the total amount of heat killed up to time $t$ within the graph.
Altogether,  $1- M_t$ is then the  amount of heat transported to the 'boundary' of the graph by the time $t$ and $M_t$  can be interpreted as the
amount of heat, which has not been transported to the boundary of the graph at
time $t$.
\smallskip

Our question concerning stochastic completeness  then becomes whether the quantity
$$ 1 - M_t$$
vanishes identically or not.  Our result reads (see  \cite{KL1} for a proof):

\begin{theorem}\label{main0} (Characterization of heat transfer to the
  boundary) Let $(V,b,c)$ be a weighted graph and $m$ a measure on $V$ of full
  support. Then, for any $\alpha >0$, the
  function
$$ w := \int_0^\infty \alpha e^{-t \alpha } ( 1 - M_t) dt$$
satisfies $0\leq w \leq 1$, solves  $(\widetilde{L} + \alpha) w = 0$, and
is the largest nonnegative function $l\leq 1$ with $(\widetilde{L} + \alpha) l
\leq  0$.  In particular, the following assertions are equivalent:
\begin{itemize}
 \item[(i)] For any $\alpha>0$ there exists  a nontrivial, nonnegative, bounded $l$ with  $(\widetilde{L} + \alpha) l \leq 0$.
\item[(ii)] For any $\alpha >0$ there exists a nontrivial, bounded $l$ with $(\widetilde{L} + \alpha) l = 0$.
\item[(iii)] For any $\alpha >0$ there exists an nontrivial, nonnegative, bounded $l$ with  $(\widetilde{L} + \alpha) l = 0$.
\item[(iv)] The function  $w$ is nontrivial.
\item[(v)] $M_t (x) <1$ for some $x\in V$ and some  $t>0$.
\item[(vi)] There exists a nontrivial, bounded, nonnegative $N : V\times
  [0,\infty)\longrightarrow [0,\infty)$ satisfying
  $\widetilde{L} N + \frac{d}{d t}  N =0$ and  $N_0 \equiv 0$.
\end{itemize}
\end{theorem}

Let us give a short interpretation of the conditions appearing in the theorem.  Conditions (i), (ii) and (iii) deal with eigenvalues of $\widetilde{L}$ considered as  an operator on $\ell^\infty (V)$. Thus, they concern spectral theory in $\ell^\infty (V)$.  Condition (v) refers to loss of mass at infinity. Finally condition (vi) is about unique solutions of a partial difference equation. Thus,  the result connects properties from stochastic processes, spectral theory and partial difference equations.

\begin{proof}[Sketch of proof]
We refrain from giving a a complete proof of the theorem but rather discuss three key elements of the proof and how they fit together. These are the following three  steps:
\begin{itemize}
\item[(S1)] If $N: [0,\infty)\times V\longrightarrow \R$ is a bounded solution of $\frac{d}{dt} N = - \widetilde{L} N$, then $v = \int_0^\infty \alpha e^{-t \alpha} N_t dt$ is a solution to $(\widetilde{L} +\alpha) v =0$ for any $\alpha >0$.
    \item[(S2)] The function $N= 1 - M$ satisfies $0\leq N\leq 1$ and $\frac{d}{dt} N = - \widetilde{L} N$.
    \item[(S3)]  The function $w = \int_0^\infty \alpha e^{-t \alpha} (1- M_t)  dt$ is the largest  solution of $(\widetilde{L} +\alpha) v =0$ with $0\leq v\leq 1$.
\end{itemize}

The proof of the first step is a direct calculation via partial integration. The second step is a direct calculation but requires quite some care as the quantities are defined via sums and integrals whose convergence is not clear. The fact that $w$ of the last step is a solution follows from the second step. The minimality of the solution requires some care. It follows by approximating the graph via finite graphs. Here, a nontrivial issue is that this approximation  may actually cut infinitely many edges (as we do not have locally finite edge degree).
\medskip

Given the three steps, the proof of the theorem goes along the following line:
The implication (v) $\Longrightarrow $ (i) follows from Step (S1) and (S2). The  implication (i) $\Longrightarrow$ (v) follows from the maximality property  in Step (3).
The implication (v) $\Longrightarrow$ (vi) follows from Step (S2). The implication (vi) $\Longrightarrow$ (v) follows from Step (S1). The equivalence between (iv) and (v) is immediate from Step (S3). The equivalence between (i), (ii) and (iii) follows by taking suitable minima of (super-) solutions.
 \end{proof}

\begin{definition} The weighted graph $(V,b,c)$ is said to satisfy $\SI$ if one of the equivalent assertions of the theorem holds. If the graph is not $\SI$ it is said to satisfy  $\SC$.
\end{definition}

In the case of vanishing killing term  (i.e. $c\equiv 0$) $\SC$ and $\SI$ are just the standard definitions of stochastic completeness and stochastic incompleteness. By a slight abuse of language we  will call any graph satisfying $\SC$ stochastically complete and any graph satisfying $\SI$ stochastically incomplete.

\begin{coro} Assume the situation of the previous theorem. Let $\widetilde{L}$ be the operator associated to the graph $(V,b,c)$. If $\widetilde{L}$ gives rise to a  bounded operator on $\ell^2(V)$, then $(V,b,c)$ satisfies $\SC$.
\begin{proof} If $\widetilde{L}$ is bounded on $\ell^2 (V,m)$ it is bounded on $\ell^\infty (V)$ by Theorem \ref{boundednesthm}. Then, the spectrum of $\widetilde{L}$ on $\ell^\infty$   is bounded and hence its set of eigenvalues is bounded as well. Thus,  (ii) of the theorem must fail (for large $\alpha$).
\end{proof}
\end{coro}

\noindent\textbf{Remark.} (a) The corollary shows that stochastic completeness is  a phenomenon for unbounded operators.

\noindent(b) The corollary  generalizes the results of Dodziuk/Matthai \cite{DM} and Wojciechowski \cite{Woj1}. It is furthermore relevant as its proof  gives an abstract i.e., spectral theoretic reason for stochastic completeness in the case of bounded operators.
\medskip

Let us finish this section by discussing how the existence of $\alpha>0$ and $t>0$ and $x\in V$ with certain properties in the above theorem is  actually equivalent to the fact that all $\alpha >0$ ,  $t>0$ and $x\in V$ have these properties. We first discuss the situation concerning the $\alpha$'s.

\begin{prop} Let $(V,b,c)$ be a weighted graph and $m$ a measure on $V$ of full  support. Then, the following are equivalent:
\begin{itemize}
\item[(i)] For any $\alpha>0$ there exists  a nontrivial, nonnegative, bounded $l$ with  $(\widetilde{L} + \alpha) l \leq 0$.
\item[(ii)]  For some  $\alpha>0$ there exists  a nontrivial, nonnegative, bounded $l$ with  $(\widetilde{L} + \alpha) l \leq 0$.
\end{itemize}
\begin{proof} It suffices to show the implication (ii) $\Longrightarrow$ (i): By the maximality  property  of the function $w =\int_0^\infty \alpha e^{-t \alpha} (1 - M_t) dt$ discussed in the third step of the proof of the main result, (ii) implies that  $M_t (x)<1$ for some $x\in V$ and $t>0$. Now, the claim (i) follows from the second step discussed in the proof of the main result.
\end{proof}
\end{prop}

We now  show that loss of mass in one point at one time is equivalent to loss of mass in all points at all times (if the graph is connected). For locally finite graphs this is discussed in \cite{Woj1}.

\begin{prop} Let $(V,b,c)$ be a connected  weighted graph and $m$ a measure on $V$ of full
  support.  Let $M$ be defined as above.  Then,  the following assertions are equivalent:
  \begin{itemize}
  \item[(i)] There exist $x\in V$ and $t>0$ with $M_t (x) < 1$.
  \item[(ii)] $M_t (x) <1$ for all $x\in V$ and all $t>0$
  \end{itemize}
\begin{proof}  The implication (ii) $\Longrightarrow$ (i) is clear. It remains to show the reverse implication. A direct calculation (invoking $\int_0^{t+s} ...dr = \int_0^s ...dr + \int_s^{t+s} ...dr$)  shows that
$$ M_{t+s}  = e^{-s L} M_t + \int_0^s e^{- r L} \frac{c}{m}\, dr.$$
This easily gives that
\begin{itemize}
\item[(1)] $M_t \equiv 1$ for some $t>0$ implies $M_{nt} \equiv 1$ for all $n\in \NN$.
\item[(2)] $M_t \neq 1$ for some $t>0$  implies that  $M_{t+s} <1$ for all $s>0$.
\end{itemize}
(Here $(1)$ follows by induction and $(2)$ follows as $M_{t} \neq 1$ implies $M_t \leq 1$ and $M_t (x)< 1$ for some $x\in V$. As the graph is connected  this implies $e^{-s L} M_t < e^{s L} 1$ and the statement follows.)

Assume now that $M_t (x) <1$ for some $x\in V$ and $t>0$.  We consider $M_r$ for  $r>t$ and for $r<t$ separately:  By $(2)$,   $M_r <1$ for all $r>t$. Assume that $M_r =1$ for some $0 <r< t$, then $M_s=1$ for all $s \leq r$ by $(2)$. Hence, by $(1)$ $M_{ns}=1$ for all $n\in \N$ and $0< s\leq r$. This gives $M_r =1$ for all $r>0$ which contradicts $M_t \neq 1$. Thus, $M_r \neq 1$ for all $0< r < t$. Hence, by (2) $M_r <1$ for all $0< r \leq t$.
\end{proof}
\end{prop}

\noindent\textbf{Acknowledgements.} It is our great pleasure to acknowledge fruitful and stimulating discussions with Peter Stollmann, Radek Wojciechowski, Andreas Weber and Jozef Dodziuk on the topics discussed in the paper.


\begin{thebibliography}{10}

\bibitem{BD}   A.~Beurling, J.~Deny.  \textit{Espaces de Dirichlet. I. Le cas \'{e}l\'{e}mentaire}. Acta Math.,  99 (1958),  203--224.

\bibitem{BD2}  A.~Beurling, J.~Deny.  \textit{Dirichlet spaces}.
 Proc. Nat. Acad. Sci. U.S.A.,   45  (1959),  208--215.


\bibitem{BH}
N.~Bouleau, F.~Hirsch. {Dirichlet forms and analysis on {W}iener space}.  Volume~14 of
  { de Gruyter Studies in Mathematics},  Walter de Gruyter \& Co., Berlin, 1991.



\bibitem{Chu} F.~R.~K.~Chung.  Spectral Graph Theory.  CBMS Regional Conference Series in Mathematics, 92,  American Mathematical Society, Providence, RI, 1997.

\bibitem{CGY} F.~R.~K.~Chung,  A.~Grigoryan, S.-T.~Yau.   \textit{Higher eigenvalues and isoperimetric inequalities on Riemannian manifolds and graphs}.  Comm. Anal. Geom.,   8  (2000),  No.~5, 969--1026.


\bibitem{Col} Y.~Colin de Verdi\`{e}re.  Spectres de graphes.   Soc. Math.  France, Paris,  1998.

\bibitem{Dav1} E.~B.~Davies.  Heat kernels and spectral theory. Cambridge
  University press, Cambridge, 1989.


\bibitem{Dav3} E.~B.~Davies. Linear operators and their spectra.
Cambridge Studies in Advanced Mathematics, 106. Cambridge University Press, Cambridge,  2007.

\bibitem{Dod0} J.~Dodziuk. \textit{Difference Equations, isoperimetric inequality and transience of  certain random walks}. Trans. Amer. Math. Soc., 284 (1984), No.~2,   787--794.

\bibitem{Dod} J.~Dodziuk. \textit{Elliptic operators on infinite graphs}.   Analysis, geometry and topology of elliptic operators,  353--368, World Sci. Publ., Hackensack, NJ, 2006.


\bibitem{DK} J.~Dodziuk, W.~S.~Kendall. \textit{Combinatorial Laplacians and isoperimetric inequality}. From local times to global geometry, control and physics (Coventry, 1984/85),  68--74, Pitman Res. Notes Math. Ser., 150, Longman Sci. Tech., Harlow, 1986.


\bibitem{DM} J.~Dodziuk, V.~Matthai. \textit{Kato's inequality and asymptotic spectral properties for discrete magnetic Laplacians}. The ubiquitous heat kernel,  69--81, Contemp. Math., 398, Amer. Math. Soc., Providence, RI, 2006.




\bibitem{Fel} W.~Feller. \textit{On boundaries and lateral conditions for the Kolmogorov differential equations}.  Ann. of Math. (2),  65  (1957),  527--570.



\bibitem{Fuj} K.~Fujiwara. \textit{Laplacians on rapidly branching trees}.  Duke Math Jour., 83 (1996), No.~1,  191-202.




\bibitem{Fuk} M.~Fukushima, Y.~Oshima, M.~ Takeda.  Dirichlet forms and symmetric Markov processes. de Gruyter Studies in Mathematics, 19. Walter de Gruyter \& Co., Berlin, 1994.



\bibitem{Gri} A.~Grigor'yan. \textit{Analytic and geometric background of reccurrence and non-explosion of the brownian motion on riemannian manifolds}.  Bull. Am. Math. Soc.,  36  (1999),  No.~2, 135--249.

\bibitem{HK} S.~Haeseler, M.~Keller, \textit{Generalized solutions and spectrum for Dirichlet forms on graphs},  to appear in to appear in Boundaries and Spectral Theory, Progress in Probability, Birkhäuser, preprint 2010,  arXiv:1002.1040.


\bibitem{HJL} O.~H{\"a}ggstr{\"o}m, J.~Jonasson, R.~Lyons. \textit{Explicit isoperimetric constants and phase transitions in the     random-cluster model}.  Ann. Probab.,  30  (2002),  No.~1, 443--473.



\bibitem{HiShi} Y.~Higuchi, T.~Shirai. \textit{Isoperimetric constants of $(d,f)$-regular planar graphs}.  Interdiscip. Inform. Sci.,  9  (2003),  No.~2, 221--228.



\bibitem{Jor} P.~E.~T.~Jorgensen. \textit{Essential selfadjointness of the graph-Laplacian}.  J. Math. Phys.,  49  (2008),  No.~7, 073510, 33p.



\bibitem{Kel} M.~Keller. \textit{The essential spectrum of Laplacians on rapidly branching tesselations}.  Math. Ann.,  346  (2010), No.~1,  51--66.


\bibitem{KL1} M.~Keller, D.~Lenz. \textit{Dirichlet forms and stochastic completeness of graphs and subgraphs}. to appear in  J. Reine Angew. Math., preprint 2009, arXiv:0904.2985.

\bibitem{KP} M.~Keller, N.~Peyerimhoff. \textit{Cheeger constants, growth and spectrum of locally tessellating planar graphs}.  to appear in Math. Z., arXiv:0903.4793.

\bibitem{Mo} B.~Mohar. \textit{Light structures in infinite planar
    graphs without the strong isoperimetric property}.
  Trans. Amer. Math. Soc.,  354 (2002), No.~8, 3059--3074.


\bibitem{MR}
Z.-M.~Ma and M.~R{\"o}ckner. Introduction to the theory of (non-symmetric) Dirichlet
  forms.  Springer-Verlag, Berlin,  1992.


\bibitem{MS} B.~Metzger, P.~Stollmann. \textit{Heat kernel estimates on weighted
  graphs}.   Bull. London Math. Soc.,  32  (2000), No.~4,  477--483.



\bibitem{Reu} G.~E.~H.~Reuter. \textit{Denumerable Markov processes and the associated contraction semigroups on $l$}.  Acta Math., 97 (1957),  1--46.

\bibitem{Stu}
K.-T.~Sturm. textit{Analysis on local Dirichlet spaces. I: Recurrence, conservativeness
  and $L\sp p$-Liouville properties}.  J. Reine Angew. Math., 456  (1994), No.~ 173--196.

\bibitem{Sto} P.~Stollmann. \textit{A convergence theorem for Dirichlet forms with applications to boundary
 value problems with varying domains}.
 Math. Z.,  219  (1995),  No. 2, 275--287.

\bibitem{SV} P.~Stollmann, J.~Voigt. \textit{Perturbation of Dirichlet forms by measures}.
 Potential Anal.  5  (1996),  No. 2, 109--138.

\bibitem{Ura} H.~Urakawa. \textit{The spectrum of an infinite graph}.
 Can. J. Math.,  52  (2000),  No.~5, 1057--1084.


\bibitem{Web} A.~Weber. \textit{Analysis of the physical Laplacian and the heat flow on
  a locally finite graph}.  J. Math. Anal. Appl.,  370  (2010),  no. 1, 146--158.


\bibitem{Woj1} R.~K.~Wojciechowski. Stochastic completeness of graphs,  PhD
thesis, 2007. arXiv:0712.1570v2.

\bibitem{Woj2} R.~K.~Wojciechowski. \textit{Heat kernel and
essential spectrum of infinite graphs}.   Indiana Univ. Math. J.,  58  (2009),  No.~3, 1419--1441.

 \bibitem{Woj3} R.~K.~Wojciechowski. \textit{Stochastically Incomplete Manifolds and Graphs}.  to appear in to appear in Boundaries and Spectral Theory, Progress in Probability, Birkhäuser, preprint 2009, arXiv:0910.5636.

\end{thebibliography}
\end{document}